\documentclass{amsart}
\usepackage{amsmath,amssymb,amsfonts,amsthm}
\usepackage{amssymb}
\usepackage{pgf}
\usepackage{todonotes}
\usepackage{epsfig}
\usepackage{mathrsfs}
\usepackage{verbatim}
\newtheorem{theorem}{Theorem}[section]
\newtheorem{lemma}[theorem]{Lemma}
\newtheorem{proposition}[theorem]{Proposition}

\theoremstyle{definition}

\theoremstyle{remark}

\numberwithin{equation}{section}

\newcommand{\diam}{\mathrm{diam}} 

\newcommand{\conv}{\mathrm{conv}}

\newcommand{\N}{\mathbb{N}}

\newcommand{\X}{\mathrm{X}}

\newcommand{\B}{\mathbf{B}}

\renewcommand{\S}{\mathbf{S}}

\newcommand{\U}{\mathcal{U}}
\renewcommand{\j}{\jmath}

\begin{document}

\title[Duality in strengths of geometric Banach space properties]{Uniform-to-proper duality of geometric properties of Banach spaces and their ultrapowers}

\begin{abstract}
In this note various geometric properties of a Banach space $\X$ are characterized by means of weaker corresponding geometric properties involving an 
ultrapower $\X^\mathcal{U}$. The characterizations do not depend on the particular choice of the free ultrafilter $\mathcal{U}$. 
For example, a point $x\in \S_\X$ is an MLUR point if and only if $\j (x)$ (given by the canonical inclusion $\j\colon \X \to \X^\mathcal{U}$) in 
$\B_{\X^\mathcal{U}}$ is an extreme point; a point $x\in \S_\X$ is LUR if and only if $\j(x)$ is not contained in any non-degenerate line segment of 
$\S_{\X^\mathcal{U}}$; a Banach space $\X$ is URED if and only if there are no $x,y \in \S_{\X^\mathcal{U}}$, $x\neq y$, with $x-y \in \j(\X)$.

\end{abstract}

\author{Jarno Talponen}

\address{University of Eastern Finland, Institute of Mathematics,
Box 111, FI-80101 Joensuu, Finland} 
\email{talponen@iki.fi}
\date{\today}

\maketitle

\section{Introduction}
This note deals with a rather general principle which connects some geometric properties of Banach spaces $\X$ to corresponding properties of the ultrapowers $\X^\U$. A well-known correspondence of this sort is the following: A Banach space $\X$ is superreflexive if and only if its ultrapower $\X^\U$ is reflexive. It is also known that a Banach space is uniformly convex if and only if its ultrapower is strictly convex. 
This is an example of a uniform-to-proper type duality mentioned in the title. Ultrafilters are more generally often used in turning `asymptotic properties' of objects to `sharp properties' of some limit objects. 
For example, F\o lner sequences in connection with amenable groups and 
$\L$os' theorem in model theory and algebra (see \cite{Folner},\cite{Los}) are such tools.

The above Banach space examples are global in a sense and here we will consider the transformation of local properties (such as a point being LUR or Fr\'echet smooth) to the corresponding local properties of the ultrapower. Recall that there is a canonical inclusion $\X \hookrightarrow \X^\U$, so that we may regard each point of $\X$ as an element of $\X^\U$ and it makes sense to 
analyze the geometry around these embedded points. It turns out that in many cases there is a clean if and only if relationship between properties of $x\in \S_\X$ and the same point 
considered in $\S_{\X^\U}$; see the abstract for examples. Some of the results here are likely folklore known to the specialists of extreme structures in Banach spaces.

Our conclusion here is that some of the geometric properties of a Banach space can even be expressed in a more attractive, concise way by means of a property involving its ultrapower instead 
(e.g. URED norm). Also recall that there are higher smoothness and rotundity properties (see e.g. \cite{sullivan}) which are expressed as corresponding properties of the bidual. 
For example, a point $x\in \S_\X$ is very smooth by definition if $x \in \X \subset \X^{**}$ is a smooth point in the bidual.  
Recall that there are linear isometric embeddings $\X \hookrightarrow \X^{**} \hookrightarrow \X^\U$, $\X \hookrightarrow \X^\U$, which commute (see \cite{Heinrich}).
Consequently, many of the geometric properties which are in force in the ultrapower $\X^\U$, are also valid in the bidual $\X^{**}$. 
Here is an illustrative example of the $\X \hookrightarrow \X^{**} \hookrightarrow \X^\U$ scale of properties: A point $x\in \S_\X$ is MLUR if and only if 
$x \in \B_{\X^\U}$ is an extreme point (see Theorem \ref{thm: first}). However, if $x$ is an extreme point (only) when considered in the bidual, then $x\in \S_\X$ is a weak-MLUR point. 
Thus, there appears to be a kind of gap between the bidual, and, apparently, the ultrapower and the `right clean' characterization of MLUR property requires 
taking an ultrapower, instead of the bidual. The results given here involving $\X^\U$ can be regarded as results on the geometry of bidual $\X^{**}$ as well.  
In the same vein, one could define apparently stronger versions of geometric properties involving higher duals, for example defining 
$x\in \S_\X$ to be a `very very smooth' point if $x$ considered in $\S_{\X^\U}$ is a Gateaux smooth point.  
As an example in the opposite direction, a uniform version of the Daugavet property can be defined by means of ultrapowers as well, see \cite{BecerraMartin}, cf. \cite{BT}. 

\subsection{Preliminaries}
We refer to \cite{FA_book}, \cite[Ch. I]{Handbook}, \cite{Heinrich} for suitable background information and notations, see also \cite{Beauzamy}. 

Let us fix some further conventions. Here $\X$ is a real Banach space, $\B_\X$ and $\S_\X$ are its unit ball and unit sphere, respectively.
In what follows,  $\U$ is a free ultrafilter on the natural numbers. The ultrapower of $\X$ with respect to $\U$ is defined as 
$\X^\U = \ell^\infty (\X) / N_\U$ where 
\[N_\U = \{(x_n ) \in \ell^\infty (\X) \colon \lim_{n,\U} \|x_n \| =0\}.\]
We denote the coset in $\X^\U$ corresponding to a sequence $(x_n ) \subset \B_\X$ by $(x_n )^\U = (x_n ) + N_\U$.
We let $\j \colon \X \to\X^\U$ be the canonical inclusion 
\[\j(x) = (x,x,x,\ldots )^\U .\]
We will distinguish the canonical inclusion $\j^* \colon \X^\ast \to (\X^* )^\U \subset (\X^\U )^*$, in case $\X^* $ is a dual space. 
The local uniform rotundity (LUR) and midpoint local uniform rotundity (MLUR) are standard and found in \cite{FA_book}. 
If the point $x \in \S_\X$ appearing in these definitions satisfies the corresponding condition separately, then $x$ is called 
a LUR point, or an MLUR point, respectively. We will use CCB short for a closed convex bounded subset of a Banach space.

\section{Some basic relationships between geometric properties between Banach spaces and their ultrapowers}

\begin{theorem}\label{thm: first}
Let $\X$ be a Banach space.
\begin{enumerate}
\item A point $x\in \S_\X$ is MLUR if and only if $\j (x) \in \S_{\X^\U}$ is an extreme point.
\item If $\j (x) \in \S_{\X^\U}$ is a G\^{a}teaux smooth point then $x\in \S_\X$ is a Fr\'echet smooth point. Moreover, if $\X$ is superreflexive, then also the converse holds. 
\end{enumerate}
In particular, the extremity of points in $\j(\S_\X )\subset \X^\U$ does not depend on the particular choice of the ultrafilter.  
\end{theorem}
\begin{proof}
To verify the first statement first, suppose that $x\in \S_\X \subset \S_{\X^\U}$ is not an extreme point in $\S_{\X^\U}$. Then there are 
$(z_n )^\U , (y_n )^\U \in \S_{\X^U}$, $\|(z_n )^\U - (y_n )^\U\|=C>0$, such that 
\[x=(x,x,x,\ldots )^\U =\frac{(z_n )^\U + (y_n )^\U}{2} .\]
Without loss of generality we may normalize the representatives in such a way that $(z_n ) ,(y_n ) \subset \S_\X$ (e.g. by letting same vectors appear finitely multiple times in the sequence, if required).
By the basic properties of the ultrapower we observe that 
\[\lim_{n,\U} \left\|\frac{x_n + y_n}{2}-x\right\|=0,\quad \lim_{n,\U} \|x_n - y_n \| =\| (z_n )^\U - (y_n )^\U \|=C.\]
Thus we may find a subsequence $(n_k )$ such that 
\[\lim_{k\to\infty} \frac{x_{n_k} + y_{n_k}}{2}=x,\quad \lim_{k\to\infty} \|x_{n_k} - y_{n_k} \| =C.\]
Thus $x\in \S_\X$ is not an MLUR point. 

Next suppose that $x\in \S_\X$ is not an MLUR point. By the assumption there are sequences $(z_n ) ,(y_n ) \subset \S_\X$ 
such that 
\[\frac{z_n +y_n }{2} \to x,\ \|z_n -y_n \|\not\to 0,\quad n\to\infty.\]
By passing to a subsequence we may assume without loss of generality that $\|z_n - y_n \| \to C>0$ as $n\to\infty$. Then
\[\lim_{n,\U} \|z_n - y_n\|=C.\]
We observe that 
\[\|(z_n )^\U - (y_n )^\U \|=C,\quad \frac{(z_n )^\U + (y_n )^\U}{2} =(x,x,x,\ldots)^\U =\j(x)\]
where $(z_n )^\U , (y_n )^\U \in \S_{\X^\U}$. Thus $\j (x) \in \S_{\X^\U}$ is not an extreme point.
        
The second part is seen in a similar fashion by applying the Smulyan lemma. If $x\in \S_\X$ and $(f_n ) , (g_n ) \subset \S_{\X^*}$ with $f_n (x) , g_n (x) \to 1$ 
and $\|f_n - g_n \|\to C>0$ (without loss of generality, as above) as $n\to\infty$, then 
$(f_n )^\U ,(g_n )^\U \in \S_{\X^\U}$ with $(f_n )^\U [\j(x)] =  (g_n )^\U [\j(x)] =1$ and $\| (f_n )^\U - (g_n )^\U \| =C$.        
Thus $\j (x) \in    \S_{\X^\U}$ is not a G\^{a}teaux smooth point.      

Recall that $(\X^*)^\U = (\X^\U )^*$ is and only if $\X$ is superreflexive (see \cite{Heinrich}) . Thus $f,g \in \S_{\X^\U}$ with $f(\j(x))=g(\j(x))=1$ 
can be written as $f:= (f_n )^\U , g:=(g_n )^\U$ with $(f_n ), (g_n ) \subset \S_{\X^*}$ and 
\[\lim_{n,\U} f_n (x) = \lim_{n,\U} g_n (x) =1. \]
If $g\neq f$ then we can find a subsequence $(n_k )$ such that $f_{n_k} (x)\to 1$, $g_{n_k} (x)\to 1 $, $\|f_{n_k} - g_{n_k}\| \to C>0$ as $k\to\infty$, 
contradicting the Fr\'echet smoothness of $x\in\S_\X$.
\end{proof}

\begin{theorem}\label{thm: lur}
A point $x\in \S_\X$ is LUR if and only if $\j(x)$ is not included in any non-degenerate line segment of $\S_{\X^\U}$.
\end{theorem}
\begin{proof}
The condition that $\j(x)$ is included in some non-degenerate line segment of $\S_{\X^\U}$ is clearly equivalent 
to the statement that there is $y\in \S_{\X^\U}$, $\|\j(x) -y\|=C>0$, with $\|\j (x)+y\|=2$. If $y = (y_n )^\U$, then we can extract a subsequence 
$(z_j ) = (y_{n_j})$ such that $\|x + z_j\|\to 2$ and $\|x - z_j\| \to C$. This case is clearly excluded if $x$ is a LUR point.

Suppose next that $x$ is not a LUR point. Then there is a sequence $(x_n ) \subset \S_\X$ such that $\lim_{n\to\infty} \|x + x_n \|=2$ 
but $\limsup_{n\to\infty} \|x- x_n \|=C>0$. Then we can extract a subsequence $(x_{n_j})$ such that $\lim_{j\to\infty} \|x-x_{n_j}\| = C$.
Then putting $(y_j )=  (x_{n_j})$ and $y=(y_j )^\U$ yields $y \in \S_{\X^\U}$ with $\|\j (x) + y\|=2$ and $\|\j(x)-y\|=C >0$.
\end{proof}

We say that a CCB set $C \subset \X$ is uniformly dentable if  there is a sequence of functionals $(f_n ) \subset \X^*$ with $f_n (C)\subset [-1,1]$ such that 
\[\sup_{\alpha_n \nearrow 1} \lim_{n\to\infty} \diam (S_{f_n , \alpha_n} (C))=0.\]
Clearly, if $C$ has a strongly exposed point, then it is uniformly dentable.

\begin{theorem}
If $x\in \S_X$ is a strongly exposed point, exposed by $f\in \S_{\X^*}$, then $\j(x)\in \B_{\X^\U}$ is an exposed point, 
exposed by $\j^* (f) \in \S_{(\X^*)^\U } \subset \S_{(\X^\U )^*}$. If $\B_\X$ is uniformly dentable, then 
$\B_{\X^\U}$ has an exposed point. 
\end{theorem}
\qed

\begin{theorem}
A Banach space $\X$ is URED if and only if there are no $x,y \in \S_{\X^\U}$, $x\neq y$, with $x-y \in \j(\X)$.
\end{theorem}
\begin{proof}
The argument is an adaptation of the proof of Theorem \ref{thm: lur}.
\end{proof}

The following result appears to be folklore and is included here for the sake of convenience. 
\begin{theorem}
The uniform convexity of $\X$ is equivalent to the strict convexity of $\X^\U$. In particular, all strictly convex ultrapowers are uniformly convex.
The uniform smoothness of $\X$ is equivalent to the G\^{a}teaux smoothness of $\X^\U$. In particular, all G\^{a}teaux smooth ultrapowers are uniformly smooth.
\end{theorem}
\qed  

As explained in the introduction, if the points of $\S_\X$ are Fr\'echet smooth regarded in $\S_{\X^\U}$, then they are also that when 
considered as elements of the bidual. It is known that if the (bi)dual of a Banach space if Fr\'echet smooth, then the space is reflexive. 
We do not know what kind of geometric properties the Fr\'echet smoothness of $\j(x)$, for all $x\in \S_\X$ implies. For instance, in the case of M-embedded space $\X$ (see e.g. 
\cite{Lima, Werner}), if any $x\in \S_\X \subset \X^{**}$ is a G\^{a}teaux smooth point, then it follows 
immediately that $X$ is reflexive.

\begin{theorem}
Let $\X$ be a superreflexive space. Then all G\^{a}teaux smooth points $x\in \S_{\X^\U}$ are in fact Fr\'echet smooth.
\end{theorem}
\begin{proof}
According to the superreflexivity of $\X$ we have $(\X^\U)^* = (\X^* )^\U$. Let $x \in \S_{\X^\U}$ be a G\^{a}teaux smooth point and let 
$f\in (\X^\U)^* = (\X^* )^\U$, $\|f\|=1$, be the unique exposing functional. Suppose that $(f_k ) \subset \S_{(\X^\U)^* }$ is a sequence with 
$f_k (x) \to 1$. 

By using the Smulyan lemma we will argue that $x$ is a Fr\'echet smooth point, thus, assume to the contrary that 
$\limsup_{k\to\infty} \|f_k - f\|=C>0$. By passing to a subsequence we may assume that 
$\|f_k -f\| \to C$.

Write $x= (x_n )^\U$, $f = (f_n )_{n}^\U$ and $f_k = (f_{k,n})_{n}^\U$ for each $k\in\N$ where $(x_n )\subset \S_{\X}$, $f_n , f_{k,n} \in \S_{\X^*}$. 
Note that 
\[\left\{n\in\N \colon \exists k\in\N\ \|f_{k,n} - f_n\|>\frac{C}{2}\right\}\in \U.\]
Define $g=(g_n )^\U$ as follows: outside the above set we set $g_n =0$. For each $n\in\N$ in the above set we choose 
$k_n \in\N$ such that 
\[ \|f_{k_n ,n} - f_n\|>\frac{C}{2},\]
\[f_{k_n , n}(x_n )> \sup \left\{f_{k,n} (x_n ) \colon k\in\N\ \mathrm{such\ that}\  \|f_{k,n} - f_n\|>\frac{C}{2}\right\} -\frac{1}{n} \] 
and we set $g_n = f_{k_n ,n}$. Clearly $g\in \S_{(\X^\U)^*}$ and $\|g- f\| \geq \frac{C}{2}$ by the construction of $g$.
Given $\varepsilon>0 $, by studying $f_k$ such that $\|f_k -f\| > \frac{C}{2}$ and $f_k (x) > 1- \varepsilon$ we conclude that 
$g(x)> 1- \varepsilon$. It follows that $g(x)=1$. Now the uniqueness of the exposing functional of $x$ fails and this provides us with a contradiction.
\end{proof}

By modifying the previous proof one can show an abstract version of the principle appearing above.
\begin{proposition}
Let $K$ be a  first-countable Hausdorff compact space. Suppose that $I$ and $J$ are index sets and that $\mathcal{F}$ is a filter on $J$ and 
$\mathcal{U}$ is a countably incomplete ultrafilter on $I$. Let $x, x_{i,j} \in K$ for all $i\in I$, $j\in J$. Assume that 
\[\lim_{j,\mathcal{F}} \lim_{i,\U} x_{i,j} = x \]
exists. Then there is a mapping $j \colon I \to J$ such that 
\[\lim_{i,\U} x_{i, j(i)} = x.\] 
\end{proposition}
\qed

\section{Resampling sequences}

Recall that a Banach space is $2\mathrm{R}$ (resp. $\mathrm{W}2\mathrm{R}$) if for each sequence $(x_n )\subset \B_\X$ with 
$\lim_{n,m \to\infty} \|x_n + x_m\|=2$ we have that $x_n \to x$ converges in the norm (resp. in the weak topology). 
The equivalent renormability of a Banach space with $\mathrm{W}2\mathrm{R}$ norm in fact characterizes reflexive spaces, see \cite{HaJo,OS}.
Also, a Banach space is reflexive if and only if it can be renormed in such a way that its bidual becomes weak-LUR, see \cite{tal}. 

We may consider graphs in Banach spaces $\X$, that is, a system $(\{x_i\}_i , \{L_{i,j} \}_{i,j} )$ of points $x_i \in \X$ and line segments 
$L_{i,j} = \conv (x_i , x_j ) \subset \X$ for some of the pairs of points $\{x_i , x_j\}$. Similarly as in graph theory, we call a graph in a Banach space complete 
if  there is a line segment between any two points. We relate to each graph the corresponding subset of the Banach space.  
We denote the complete graph $\Gamma_{\{x_i \}_i} = \bigcup \{\conv(x,y)\colon x,y \in \{x_i \}_i \}$ (considered a set).

Given a sequence $(x_n )\subset \B_\X$, we denote its `resampling ultraset' as follows: 
\[R_{(x_n )}:=\left\{(y_n )^\U \colon (y_n)=(x_{\pi (n)}),\  \pi \colon \N \to \N\ \mathrm{bijection}\right\} \subset \X^\U .\]
It is easy to see that this set depends on the particular representative $(x_n )\in (z_n )^\U$.
The question about the cardinality of the set can be approached by means of Rudin-Keisler ordering of ultrafilters.

\begin{proposition}\label{prop: incl}
Let $(x_n )\subset \B_\X$ be a sequence. Then $\lim_{n,m\to\infty} \|x_n +x_m \|=2$ if and only if $\Gamma_{R_{(x_n )}} \subset \S_{\X^\U}$.
\end{proposition}
\begin{proof}
Fix $(x_n )\subset \B_\X$. Suppose $\lim_{n,m\to\infty} \|x_n +x_m \|=2$. Let $(y_n )$ and $(z_n )$ be sequences obtained by permuting $(x_n )$. 
Then $\lim_{n, \U} \|y_n + z_n\| =2$ according to the assumption, so that $\frac{y + z}{2} \in \S_{\X^\U}$, where
$y =  (y_n )^\U$ and $ z = (z_n )^\U $. By the convexity of $\B_{\X^\U}$ we get that $\conv (y, z) \subset \S_{\X^\U}$.
This proves the `only if' direction.

Next suppose that $\lim_{n,m\to\infty} \|x_n +x_m \|=2$ does not hold. Then we can extract subsequences $(n_j )$ and $(m_j )$ such that 
$\|x_{n_j} + x_{m_j }\| \to C <2$ as $j\to\infty$. Let $V$ and $W$ be the sets of even and odd natural numbers, respectively.  
Thus one of these set is in $\U$ and the other one is not, say, $W \notin \U \ni V$ without loss of generality (by permuting $\N$ suitably).
Fix permutations $\pi_a$ and $\pi_b$ of the natural numbers such that $\pi_a (2j) = (n_j )$, 
$\pi_b (2j ) = (m_j )$. Then $n_j = \pi_a \pi_{b}^{-1} (m_j )$ for each $j\in\N$. This means that 
for $(y_n )=(x_{\pi_a (n)})$ and $(z_n )=(x_{\pi_a (n)})$ we have 
$\lim_{n\to \infty} \|y_{2n} + z_{2n}\|=C$. Since $V \in \U$ we have $\lim_{n, \U} \|y_{n} + z_{n}\|=C$. Put 
$y=(y_n )^\U$ and $z=(z_n )^\U$. Then $y,z \in R_{(x_n )}$ but $\|\frac{y + z}{2}\| = \frac{C}{2}<1$. Therefore the 
complete graph of $R_{(x_n )}$ is not included in $\S_{\X^\U}$, showing the `if' direction.
\end{proof}

\begin{theorem}\label{thm: rotund_2R}
A Banach space $\X$ is $2\mathrm{R}$ if and only if for any 
$(x_n ) \subset \B_\X$ the inclusion $\Gamma_{R_{(x_n )}} \subset \S_{\X^\U}$  implies that $\Gamma_{R_{(x_n )}} $ is a singleton.
\end{theorem}
This characterization of $2\mathrm{R}$ is clearly some kind of rotundity condition on $\B_{\X^\U}$ and it is `local' in the sense that 
it involves only a single countable subset of $\X$, one at a time. 

The following auxiliary tool appears to have Ramsey theoretic flavor.
\begin{lemma}\label{lm: alt}
Let $(X,d)$ be a metric space and $(x_n ) \subset X$ a sequence. Then exactly one of the following conditions hold:
\begin{enumerate}
\item There is $\alpha>0$ and a subsequence $(x_{n_j})$ such that $d(x_{n_{j_1}},x_{n_{j_2}})>\alpha$ for each $j_1 \neq j_2$.
\item Any subsequence $(n_{j}) \subset \N$ contains a further subsequence $(n_{j_i})_i \subset (n_j)$ such that  $(x_{n_{j_i}})_i$ is Cauchy.
\end{enumerate}
\end{lemma}
\begin{proof}
Clearly the above conditions are mutually exclusive.

We call a sequence $(y_n) \subset X$ $\alpha$-Crudely Cauchy, $\alpha>0$,  if 
\[\limsup_{n\to\infty} \limsup_{k \to\infty} d(y_n , y_{n+k}) < \alpha.\]
By a diagonal argument we obtain that  $(y_n )$ contains a Cauchy subsequence if it is $\alpha$-Crudely Cauchy for every $\alpha>0$. 
We proceed in $2$ cases. 

First assume that for each subsequence $(x_{n_j})$ and $\alpha>0$ there is a further subsequence $(x_{n_{j_i}})$ which is 
$\alpha$-Crudely Cauchy. Put $\alpha_n = \frac{1}{n}$ for $n\in\N$. Let $(n_{j})$ be as in condition 2.
According to the above assumption we obtain that $(x_{n_{j}})$ contains a subsequence 
which is $\alpha_1$-Crudely Cauchy. Moreover, from each $\alpha_n$-Crudely Cauchy subsequence we may 
pass on to a further subsequence which is $\alpha_{n+1}$-Crudely Cauchy.  Then a standard diagonal argument yields a subsequence 
$(z_k) \subset (x_n)$ which is $\alpha_n$-Crudely Cauchy for all $n\in\N$. That is, $(z_k ) \subset (x_{n_{j}})$ contains a Cauchy 
subsequence.

Assume next that there is a subsequence $(x_{n_j})$ and $\alpha_0 >0$ such that no further subsequence of $(x_{n_j})$ is $\alpha_0$-Crudely Cauchy. 
This means that there is a subsequence $(x_{n_{j_k}})$ such that $d(x_{n_{j_1}} , x_{n_{j_k}})>\frac{\alpha_0}{2}$ for each $k >1$. 
Because of the assumption we may extract a further subsequence $(x_{n_{j_{k_\ell}}})$ such that $x_{n_{j_{k_1}}}=x_{n_{j_1}} $ and 
$d(x_{n_{j_{k_2}}}, x_{n_{j_{k_\ell}}})>\frac{\alpha_0}{2}$ for $\ell>2$. Proceeding in this manner by a standard diagonal argument 
we obtain a subsequence $(z_n)$ which satisfies condition 1. for $\alpha=\frac{\alpha_0}{2}$.
\end{proof}

\begin{proposition}\label{prop: char}
Let $\X$ be a Banach space and $(x_n )\subset \B_\X$ a sequence. 
Then $(x_n )$ is norm-convergent (resp. weak-star convergent in case $\X$ is a dual space) if and only if for each pair of subsequences $(n_k ), (n_m) \subset \N$ 
there are further subsequences $(n_{k_i} ) \subset (n_k )$ and $(n_{m_i}) \subset (n_m)$ such that 
$x_{n_{k_i}} - x_{n_{m_i} } \to 0$ in norm topology (resp. weak-star topology) as $i\to\infty$.
\end{proposition}

The analogous result does not typically hold for the weak topology in the non-reflexive case.

\begin{proof}
The `only if' direction is clear for both topologies. 

To prove the `if' part for the weak-star topology, assume that $(x_{n}^* )\subset \B_{\X^*}$ is not weak-star convergent. 
By the weak-star compactness of the dual unit ball we obtain $2$ weak-star cluster points, say $z^{*} \neq y^{*}$. Indeed, if there were 
only weak-star cluster point, then $(x_{n}^* )$ would weak-star converge in the first place. These cluster points can be separated by 
a functional $x \in \X$ in such a way, say, that $0 \leq x(z^{*} ) \leq 1$ and $x(y^{*}) =a>1$, without loss of generality. 
Note that 
\[z^{*} \notin \overline{\left\{x_{n}^* \colon x(x_{n}^* )> \frac{2a+1}{3}\right\}}^{\omega^*} \ni y^{*} ,\quad 
y^{*} \notin \overline{\left\{x_{n}^* \colon x(x_{n}^* )< \frac{a+2}{3}\right\}}^{\omega^*} \ni z^{*} .\]
Then we select subsequences $(x_{n_k}^* )$ and $(x_{n_m}^* )$ according to the above sets. It is now clear that there are no further subsequences 
with $x_{n_{k_i}}^*  - x_{n_{m_i} }^* \to 0$ in the weak-star topology. 
  
To prove the `if' part for the norm topology, we will use the alternative provided by Lemma \ref{lm: alt}. First suppose that there is 
a subsequence $(z_n ) \subset (x_n)$ which is a discrete set, as in the Lemma. Then we define $(x_{n_k})_k = (z_{2k})_k$ and $(x_{n_m})_m = (z_{2m+1})_m$.
Then the assumption of the statement of the proposition is not valid. 

Consequently, we are required to only study the case where each subsequence of $(x_n)$ contains a Cauchy sequence. Let 
$(x_{n_m} ) \subset (x_n )$ be a Cauchy sequence. The $x_{n_m} \to z$ in the norm topology as $m\to\infty$. Suppose that $x_n \not\to z$, then 
$\limsup_{n\to\infty} \|x_n - z\| =b>0$. Then we may isolate a subsequence $(x_{n_k} ) \subset (x_n)$ such that 
$\|x_{n_k} - z\|> \frac{b}{2}$ for each $k\in\N$. Therefore it is not possible to isolate further subsequences $(n_{m_i})$ and $(n_{k_i})$ 
such that $\lim_{i\to\infty} \|x_{n_{k_i}} - x_{n_{m_i} } \|=0$ holds.
\end{proof}

\begin{proof}[Proof of Theorem \ref{thm: rotund_2R}]
Fix $(x_n ) \subset \B_\X$. Assume $\Gamma_{R_{(x_n )}}\subset \S_{\X^\U}$, or equivalently 
$\lim_{n,m\to\infty} \|x_n + x_m\|=2$, according to Proposition \ref{prop: incl}. Clearly $\|x_n \| \to 1$.
Assume that $\X$ is $2\mathrm{R}$. Then there is $x \in \S_\X$ such that $x_n \to x$ in the norm. Note that 
every permutation of $(x_n )$ converges in norm to $x$ as well. This means that $R_{(x_n )}$ consists of one element only, namely 
$(x,x,x,\ldots )^\U$. 

Next assume that the condition $2\mathrm{R}$ of $\X$ fails and that $(x_n )$ does not converge in the norm. 
Assume to the contrary that $\Gamma_{R_{(x_n )}}$ is a singleton.
Then according to Proposition \ref{prop: char} there are subsequences 
$(n_k),(n_m) \subset \N$ such that there do not exists further subsequences $(n_{k_i})$, $(n_{m_i})$ with 
$\|x_{n_{k_i}} - x_{n_{m_i} }\| \to 0$. 

Let $(y_k)=(x_{n_k})$ and $(z_m)=(x_{n_m})$. Put $y = (y_k)^\U$ and $z = (z_m)^\U$. Observe that 
\[\lim_{n, \U} \|y_n + z_n\|=\lim_{n,m\to \infty} \|x_n + x_m\|=2,\]
\[\lim_{n, \U} \|y_n\| =1,\quad \lim_{n, \U} \|z_n\| =1 .\]
According to the assumption that $\Gamma_{R_{(x_n )}}$ is a singleton we obtain that $z=y$. Thus 
\[\lim_{n, \U} \|y_n -z_n\|=0.\]
This means that there is a subsequence $(n_i) \subset \N$ such that $\|y_{n_i} - z_{n_i} \|\to 0$ as $i\to\infty$. This can be rephrased as follows:
$\|x_{n_{k_i}} - x_{n_{m_i}} \|\to 0$ as $i\to\infty$, which contradicts the selection of the sequences $(n_k)$ and $(n_m)$. 
This completes the proof.
\end{proof}

\subsection*{Acknowledgments}
The author is grateful to Petr H\'ajek for inspiring discussions on the topic.
This research was financially supported by the Academy of Finland Project \#268009, the Finnish Cultural Foundation and V\"{a}is\"{a}l\"{a} Foundation.


\begin{thebibliography}{99}
\bibitem{Werner}
D. Werner, New classes of Banach spaces which are M-ideals in their biduals
\end{thebibliography}
\end{document}